\documentclass[11pt]{article}
\usepackage[applemac]{inputenc}
\usepackage[english]{babel}
\usepackage{amssymb}
\usepackage{amsthm}
\usepackage{stmaryrd}
\usepackage{amsmath}
\newtheorem{de}{Definition}[section]
\newtheorem{prop}{Proposition}[section]
\newtheorem{theo}{Theorem}[section]

\newtheorem{coro}{Corollary}[section]
\newtheorem{rem}{Remark}[section]

\DeclareMathOperator{\Hom}{Hom}

\DeclareMathOperator{\Rep}{Rep}

\begin{document}
\title{On the interior motive of certain Shimura varieties : the case of Picard varieties}
\author{Guillaume Cloître}
\date{}
\maketitle

\begin{abstract}

The aim of this article is the construction of the interior motive of a Picard variety. Those are Shimura varieties of PEL type. Our result is an application of the strategy developed by Wildeshaus to construct a Hecke-invariant motive whose realizations correspond to interior cohomology.

\end{abstract}

\tableofcontents

\section{Introduction}

Picard varieties are Shimura varieties of PEL type. Their construction involves a CM field $E$. They admits an interpretation as moduli spaces of some abelian varieties of dimension $3g$ with additional data, $g$ being the degree of the maximal sub-field totally real of $E$. 

A Picard variety $S$ admits a Baily-Borel compactification $S^*$ obtained by adding a finite number of cusps. 

In \cite{Wil08}, a strategy is developed to construct the so-called interior motive of some variety $X$. The interior motive is a motivic counter-part of the interior cohomology. It is also some kind of an minimal compactification for the motive of $X$.

The strategy we follow relies on the notion of weights applied to the so-called boundary motive $\partial M_{gm}(X)$. By definition, the later is part of an exact triangle
$$\partial M_{gm}(X)\rightarrow M_{gm}(X)\rightarrow M_{gm}^c(X)\rightarrow \partial M_{gm}(X)[1].$$
The morphism $M_{gm}(X)\rightarrow M_{gm}^c(X)$ we are interested in is in some sense controlled by the boundary motive.

The triangulated category of motives admits a weight structure which permits to identify motives of proper varieties (as are compactifications) and their direct factors with objects of weight zero. 

The notion of weights lacks functoriality in some sense. Nevertheless, the avoidance of certain weights solve this problem. As a consequence, a condition of avoidance of certain weights by $\partial M_{gm}(X)$ permits the construction of the interior motive.

The motives we are studying are of specific type, namely their are abelian. It allows to reduce the verification of avoidance of weights to the analog property on the realizations.

Thus, the main goal of this article is to calculate the weights of some degeneration of Hodge structure on a Picard variety $S$ in order to construct the interior motive of some direct factor of the motive associated to Kuga-Sato $A^r$ families over $S$. As in previous cases (e.g. \cite{Wil09} and \cite{Wil14}), we found that a condition of regularity of a character defining our direct factor allows for the construction. We study also the reverse statement, namely equivalence between avoidance of weights and regularity condition.

Eventually, we use those results to associate a motive to some automorphic forms of Picard varieties. \newline

The first part studies Picard varieties. The second part contains the main result, namely the calculation of weights appearing in some degenrations of Hodge structures and the construction of the interior motive of some direct factors of the Kuga-Sato family $A^r$. The third part deals with motives associated to automorphic forms.\newline

\textbf{Notations: }Let $k$ be a perfect field. We denote by $Sch/k$ the category of separated schemes of finite type over $k$ and $Sm/k$ the full sub-category of smooth schemes over $k$.

$DM_{gm}(k)$ will be the triangulated category of geometric motives over $k$ (defined for example in \cite{LBF}). When $Q$ is a commutative $\mathbb{Q}$-algebra, $DM_{gm}(k)_Q$ will be the version with $Q$-coefficients.

$CHM(k)$ is the category of Chow motives defined in \cite{LBF}. It is the opposite category of the one defined in \cite{And}. $CHM(k)_Q$ is the $Q$-coefficient version, for $Q$ a commutative $\mathbb{Q}$-algebra.

For $S\in Sm/k$, we denote by $CHM^s(S)$ the category of relative Chow motives which is the opposite of the one defined in section 1.6. of \cite{DM}. There is also a $Q$-coefficient version $CHM^s(S)_Q$ for $Q$ a commutative $\mathbb{Q}$-algebra.

Finally, for $x$ a real number, $[x]$ denote its floor.

\section{Picard varieties}

Let $E$ be a CM field, $F$ its maximal totally real subfield, $\mathcal{O}_E$ the integral ring of $E$ and $\alpha$ the element of order two in $Aut(E/F)$. We denote by $g=[F:\mathbb{Q}]$ the degree of $F$ over $\mathbb{Q}$. We choose $\Phi=\{\sigma_i\}_{1\leq i\leq g}$ a type of $E$, which means that the set of complex embeddings of $E$ in $\mathbb{C}$ is $\Phi\cup\overline{\Phi}$.

Let $V$ be a three dimensional vector space over $E$, $J$ an hermitian form on $V$ such that for every complex embedding $\sigma:E\rightarrow\mathbb{C}$, the hermitian form $J_\sigma:V_\sigma\times V_\sigma\rightarrow\mathbb{C}$ has signature $(2,1)$ where $V_\sigma=V\otimes_{E,\sigma}\mathbb{C}$.

On consider also an $\mathcal{O}_E$-lattice $L$ such that $J(L,L)\subset\mathcal{O}_E$.

\begin{prop}
There exists a basis $\mathfrak{B}$ of $V$ in which $J$ has the form
$$J_b=\left(\begin{array}{ccc} 0 & 0 & 1 \\ 0 & b & 0 \\ 1 & 0 & 0\end{array}\right)$$ 
for a certain $b\in F$.
\end{prop}
\begin{proof}
From \cite{O}, we know that there exists an isotropic vector. The rest follows from basic linear algebra.
\end{proof}

\begin{de}
Let $\tilde{G}:=Res_{F/\mathbb{Q}}GU(V;J)_F$, that is to say that for every $\mathbb{Q}$-algebra $R$, we have 
$$\tilde{G}(R)=\{g\in GL_{E\otimes_\mathbb{Q} R}(V\otimes_\mathbb{Q}R)/\exists \lambda(g)\in (E\otimes_\mathbb{Q} R)^\times ,J(g.,g.)=\lambda(g)J(.,.)\}.$$

We denote $\lambda:\tilde{G}\rightarrow Res_{F/\mathbb{Q}}\mathbb{G}_{m,F}$ the morphism associated to the multiplier.

Let $G:=\tilde{G}\times_{Res_{F/\mathbb{Q}}\mathbb{G}_{m,F}}\mathbb{G}_{m,\mathbb{Q}}$, the fiber product being associated to the multiplier and the canonical morphism $\mathbb{G}_{m,\mathbb{Q}}\rightarrow Res_{F/\mathbb{Q}}\mathbb{G}_{m,F}$.
\end{de}

The isomorphism $E\otimes_\mathbb{Q}\mathbb{C}\simeq(\mathbb{C}\times\mathbb{C})^{g}$ induces an isomorphism 
$$V\otimes_{\mathbb{Q}}\mathbb{C}\simeq \prod\limits_{i=1}^g(V_{i,+}\times V_{i,-}).$$ 
$V_{i,+}$ (resp. $V_{i,-}$) is the subspace where $x\in E$ acts as $\sigma_i(x)$ (resp. $\overline{\sigma_i}(x)$). It induces in turn an isomorphisme $\tilde{G}_\mathbb{C}\simeq GL_{3,\mathbb{C}}^g\times\mathbb{G}_{m,\mathbb{C}}^g.$, hence an isomorphisme
$$G_\mathbb{C}\simeq GL_{3,\mathbb{C}}^g\times\mathbb{G}_{m,\mathbb{C}}.$$

Complex conjugation acts on $\tilde{G}(\mathbb{C})$ as $(g_i,\lambda_i)_{1\leq i\leq g}\mapsto(\overline{\lambda_i}J^{-1}{}^tg_i^{-1}J,\overline{\lambda_i})_{1\leq i\leq g}$. Thus, we have
$$G(\mathbb{R})=\{((g_i)_{1\leq i\leq g},\lambda)\in GL_{3,\mathbb{C}}^g\times \mathbb{R}^\times/\forall 1\leq i\leq g, \lambda g_i^*=g_i\}.$$

The center $Z(G)$ of $G$ is
$$Z(G)=Res_{E/\mathbb{Q}}\mathbb{G}_{m,E}\times_{Res_{F/\mathbb{Q}}\mathbb{G}_{m,F}}\mathbb{G}_{m,\mathbb{Q}}.$$

We define a morphism $h:\mathbb{S}\mapsto G_\mathbb{R}$ given on $\mathbb{C}$ by
$$(z_1,z_2)\mapsto ((\left(\begin{array}{ccc} \frac{z_1+z_2}{2} & 0 & \frac{z_1-z_2}{2} \\ 0 & z_1 & 0 \\ \frac{z_1-z_2}{2} & 0 & \frac{z_1+z_2}{2}\end{array}\right))_{1\leq i\leq g},z_1z_2).$$

This morphism and the set $X$ of its $G(\mathbb{R})$-conjugates determine a Shimura datum $(G,X)$. We choose $K$ an open compact sub-group of $G(\mathbb{A}_f)$ which stabilises $L$. We associates to it the Picard variety $S(G,K)(\mathbb{C})=G(\mathbb{Q})\backslash X\times G(\mathbb{A}_f)/K$, denoted $S$ when the context is clear. It admits an interpretation as a moduli space of abelian varieties of dimension $3g$ which additional data (see \cite{Gor} and \cite{Cl} for a precise statement). We will consider such groups which are also neat. In this case, the Picard variety is smooth. We denote by $A\rightarrow S$ the universal abelian variety. $A^r$ will be the $r$-fold fiber product of $A$ over $S$. It is the Shimura variety associated to the Shimura datum given by the unipotent extension of $G$ by $V^r$ (see \cite{Pin}).

\begin{rem}
In \cite{Gor}, the Shimura data is $(\tilde{G},X)$ but it does not fit the condition (+) of \cite{BW}.
\end{rem}

A Picard variety admits a model over its reflex field $E(G,X)$ which is the field of definition of the $G(\mathbb{C})$-conjugation class of the morphism $\mu:\mathbb{G}_{m,\mathbb{C}}\rightarrow G_\mathbb{C}$ given by
$$z\mapsto ((\left(\begin{array}{ccc} \frac{z+1}{2} & 0 & \frac{z-1}{2} \\ 0 & z & 0 \\ \frac{z-1}{2} & 0 & \frac{z+1}{2}\end{array}\right))_{1\leq i\leq g},z).$$

Let $\mathbb{Q}^{al}$ be an algebraic closure of $\mathbb{Q}$ containing $E$. We recall that the reflex field of the type $(E,\Phi)$ is the subfield of $\mathbb{Q}^{al}$ fixed by the elements $\tau\in Gal(\mathbb{Q}^{al}/\mathbb{Q})$ such that $\tau(\Phi)=\Phi$ (see \cite{Mil}). 

\begin{prop}
The reflex field of $(G,X)$ is the reflex field of $(E,\Phi)$.
\end{prop}
\begin{proof}

Let $E_{gal}$ be a Galois closure of $E$ in $\mathbb{Q}^{al}$. $G$ splits over $E_{gal}$, thus $\mu$ is defined over it.

Let $\tau\in\text{Gal}(E_{gal}/\mathbb{Q})$. It acts on $G_{E_{gal}}$. We write $\tau^{-1}(i)$ the integer such that $\sigma_{\tau^{-1}(i)}\in\{\tau^{-1}\sigma_i,\tau^{-1}\overline{\sigma_i}\}$. We denote by $\tilde{g}_{\tau^{-1}(i)}$ the element which is $g_{\tau^{-1}(i)}$ (resp. $\lambda J^{-1}{}^tg_{\tau^{-1}(i)}^{-1}J$) if $\tau^{-1}\sigma_i=\sigma_{\tau^{-1}(i)}$ (resp. si $\tau^{-1}\sigma_i=\overline{\sigma_{\tau^{-1}(i)}}$).

Then the action of $\tau$ on $(g_i,\lambda)\in G_{E_{gal}}$ is given by
$$\tau.((g_i)_{1\leq i\leq g},\lambda)=((\tau.\tilde{g}_{\tau^{-1}(i)})_{1\leq i\leq g},\tau.\lambda).$$

As the $g$ projections $G_{E_{gal}}\rightarrow GL_{3,E_{gal}}$ composed with $\mu$ are identical, it shows that $\mu$ is defined over the reflex field of $(E,\Phi)$.

What is more, $V_{i,+}$ (resp. $V_{i,-}$) decomposes under the action of $\mathbb{G}_{m}$ via $\mu$ (or one of its conjugate) in a space of weight $0$ of dimension 1 (resp. 2) and a space of weight $-1$ of dimension 2 (resp. 1). An automorphism $\tau\in Gal(E_{gal}/\mathbb{Q})$ exchanging $\sigma_i$ and $\overline{\sigma_i}$ exchanges also $V_{i,+}$ and $V_{i,-}$. Therefore, if a conjugate of $\mu$ were defined over a field fixed by such an automorphism, the space of weight 0 of $V_{i,+}$ would be exchange with its subspace of weight $-1$ (and the same would be true for $V_{i,-}$).

Therefore, any conjugate of $\mu$ can not be defined over such a field. Thus, the field of definition of the conjugate class of $\mu$ is the reflex field of $(E,\Phi)$.

\end{proof}

We know from \cite{Pin} that the Picard variety $S$ possesses a canonical model over its reflex field. The same is true for the universal abelian variety $A$ over $S$. For the determination of the reflex field of $(E,\Phi)$, we refer to \cite{Dod}.\newline

We are now going to determine compactifications of $S$ following \cite{Pin}. First we need to know Borel subgroups and maximal tori.

\begin{prop}
Every parabolic subgroup $Q$ of $G$ defined over $\mathbb{Q}$ is conjugated in $G$ to a group of the form
$$(\left(\begin{array}{ccc} * & * & * \\ 0 & * & * \\ 0 & 0 & * \end{array} \right), *)\cap G.$$
Such a group is automatically a Borel sub-group.

The unipotent radical of $Q$ has dimension $3g$.

The maximal tori defined over $\mathbb{C}$ have dimension $3g+1$ and are conjugated to 
$$T_m=(\left(\begin{array}{ccc} * & 0 & 0 \\ 0 & * & 0 \\ 0 & 0 & * \end{array} \right)_{1\leq i\leq g}, *)$$

The maximal split tori defined over $\mathbb{Q}$ have dimension $2$ and are conjugated to 
$$T=\{(\left(\begin{array}{ccc} t & 0 & 0 \\ 0 & r & 0 \\ 0 & 0 & r^2t^{-1}\end{array}\right)_{1\leq i\leq g},r^2)\}.$$
\end{prop}
\begin{proof}

The first subgroup described is a Borel subgroup. Its unipotent radical is 
$$(\left(\begin{array}{ccc} 1 & * & * \\ 0 & 1 & * \\ 0 & 0 & 1 \end{array} \right), 1)\cap G.$$
Considering a $\mathbb{Q}$-basis of $E$ and writing explicitly the equations defining $G$, we see that its dimension is $3g$.

Seeing $G$ as a subgroup of $Res_{E/\mathbb{Q}}GL_{3,E}$, we know that a $\mathbb{Q}$-parabolic must stabilize a subspace $W$ of $V$. Then it stabilises also $W^\bot$, the $J$-orthogonal of $W$. As a parabolic contains a Borel, $W$ and $W^\bot$ can not have zero intersection for dimension reason. Thus, $Q$ stabilises an isotropic line. As the action of $G$ on isotropic lines is transitive, we get the result for parabolic subgroups.

The assertion for complex tori is clear. Inside such a torus, we see that the given torus is a maximal rational torus.

\end{proof}

We recall some conventions used in \cite{Pin}. We consider $H_0=GL_{2,\mathbb{R}}\times_{\mathbb{G}_{m,\mathbb{R}}}\mathbb{S}$, the morphisms defying the fiber product being the determinant and the norm.

We use the morphism $h_0: \mathbb{S}\rightarrow H_0$ et $h_\infty : \mathbb{S}_\mathbb{C}\rightarrow H_{0,\mathbb{C}}$ defined by :
$$h_0(z_1,z_2)=(\left(\begin{array}{cc} \frac{z_1+z_2}{2} & \frac{z_1-z_2}{2i} \\ \frac{z_2-z_1}{2i} & \frac{z_1+z_2}{2}\end{array}\right),(z_1,z_2))$$
$$h_\infty(z_1,z_2)=(\left(\begin{array}{cc} z_1z_2 & i(1-z_1z_2) \\ 0 & 1\end{array}\right),(z_1,z_2)).$$

Finally, we denote by $w:\mathbb{G}_{m,\mathbb{R}}\rightarrow\mathbb{S}$ the canonical morphism.

\begin{prop}
Let $(Q,T)$ be as in the previous proposition. Then the morphism $\lambda : \mathbb{G}_{m\mathbb{Q}}\rightarrow T$ defined by
$$t\mapsto (\left(\begin{array}{ccc} t & 0 & 0 \\ 0 & 1 & 0 \\ 0 & 0 & t^{-1} \end{array} \right)_{1\leq i\leq g},1)$$ 
is the character associated to (Q,T) in \cite{Pin}.
\end{prop}
\begin{proof}
The only simple root associated to $(Q,T)$ is
$$\alpha_0 :(\left(\begin{array}{ccc} t & 0 & 0 \\ 0 & r & 0 \\ 0 & 0 & r^2t^{-1} \end{array}\right),r^2)\mapsto tr^{-1}.$$

The weights appearing in $\text{Lie }Q$ are $\alpha_0$ and $2\alpha_0$. The fundamental weight associated to $\alpha_0$ is
$$(\left(\begin{array}{ccc} t & 0 & 0 \\ 0 & r & 0 \\ 0 & 0 & r^2t^{-1} \end{array}\right),r^2)\mapsto t.$$

The character associated to it is
$$\lambda_0:t\mapsto (\left(\begin{array}{ccc} t & 0 & 0 \\ 0 & 1 & 0 \\ 0 & 0 & t^{-1} \end{array} \right)_{1\leq i\leq g},1).$$

\end{proof}

\begin{prop}
Let $\omega:H_{0,\mathbb{C}}\rightarrow G_\mathbb{C}$ be the morphism defined by 
$$(\left(\begin{array}{cc} a & b \\ c & d \end{array}\right),(z_1,z_2))\mapsto (\left(\begin{array}{ccc} a & 0 & ib \\ 0 & z_1 & 0 \\ -ic & 0 & d \end{array}\right)_{1\leq i\leq g},z_1z_2).$$
Then $\omega$ is the morphism of proposition 4.6 of \cite{Pin}
\end{prop}
\begin{proof}
We need to check the three points of proposition 4.6. of \cite{Pin}. 

For (i), we see that for $a,b,c,d\in\mathbb{C}$, $z_1,z_2\in\mathbb{C}^\times$ such that $ad-bc=z_1z_2$, we get 
$$\left(\begin{array}{ccc} a & 0 & ib \\ 0 & z_1 & 0 \\ -ic & 0 & d \end{array}\right)^{-1}=\left(\begin{array}{ccc} \frac{d}{z_1z_2} & 0 & \frac{-ib}{z_1z_2} \\ 0 & z_1^{-1} & 0 \\ \frac{ic}{z_1z_2} & 0 & \frac{a}{z_1z_2} \end{array}\right).$$
Thus 
$$\overline{z_1z_2}\left(\begin{array}{ccc} a & 0 & ib \\ 0 & z_1 & 0 \\ -ic & 0 & d \end{array}\right)^*=\left(\begin{array}{ccc} 0 & 0 & 1 \\ 0 & b^{-1} & 0 \\ 1 & 0 & 0 \end{array}\right)\left(\begin{array}{ccc} \overline{d} & 0 & -i\overline{c} \\ 0 & \overline{z_2} & 0 \\ i\overline{b} & 0 & \overline{a} \end{array}\right)\left(\begin{array}{ccc} 0 & 0 & 1 \\ 0 & b & 0 \\ 1 & 0 & 0 \end{array}\right)$$
$$=\left(\begin{array}{ccc} \overline{a} & 0 & i\overline{b} \\ 0 & \overline{z_2} & 0 \\ -i\overline{c} & 0 & \overline{d} \end{array}\right)$$
which shows that $\omega$ is well defined over $\mathbb{R}$. What is more, $\omega\circ h_0=h$ which shows (ii).\newline

For (iii), we have 
$$\omega\circ h_\infty\circ w=t\mapsto(\left(\begin{array}{ccc} t^2 & 0 & t^2-1 \\ 0 & t & 0 \\ 0 & 0 & 1\end{array}\right),t^2)$$
which is conjugated to 
$$\lambda. h\circ w=t\mapsto(\left(\begin{array}{ccc} t^2 & 0 & 0 \\ 0 & t & 0 \\ 0 & 0 & 1\end{array}\right),t^2)$$
by
$$(\left(\begin{array}{ccc} 1 & 0 & -1 \\ 0 & 1 & 0 \\ 0 & 0 & 1\end{array}\right),1)\in Q_\mathbb{C}.$$
In addition, $Lie(Q)_\mathbb{C}$ is the sum of non-negative weight spaces of $Lie(G)_\mathbb{C}$ for the action of $Ad_G\circ\omega\circ h_\infty\circ w$.
\end{proof}

We consider the torus $T_1$ included in $G$ such that, for every $\mathbb{Q}$-algebra $R$,
$$T_1(R)=\{(\left(\begin{array}{ccc} z\alpha(z) & 0 & 0 \\ 0 & z & 0 \\ 0 & 0 & 1\end{array}\right)\,,z\alpha(z))\,,\,z\in (E\otimes_\mathbb{Q} R)^\times\text{ et }z\alpha(z)\in R^\times\}.$$

We have 
$$T_1=Res_{E/\mathbb{Q}}\mathbb{G}_{m,E}\times_{Res_{F/\mathbb{Q}}\mathbb{G}_{m,F}}\mathbb{G}_{m,\mathbb{Q}}$$ 

\begin{prop}
The smallest normal subgroup $P$ of $Q$ containing the image of $\mathbb{S}$ par $\omega\circ h_\infty$ is $P=W.T_1$. Its unipotent radical is the same as the unipotent radical of $Q$. In addition, $P(\mathbb{R})$ is connected. 
\end{prop}
\begin{proof}
The proof of lemma 4.8. of \cite{Pin} contains the assertion about unipotent radicals.

As a consequence, the smallest normal subgroup of $Q$ containing the image of $\omega\circ h_\infty$ must contain $W$. It must also contain $T_1$. Thus, $W.T_1\subset P$. 

As $W.T_1$ is normal in $Q$ and contains the image of $\omega\circ h_\infty$, we have 
$$P=W.T_1.$$

Eventually, $P(\mathbb{R})$ is generated by $W(\mathbb{R})$ which is connected and the group $$T_1(\mathbb{R})=\{(z_i)_{1\leq i\leq g}\in\mathbb{C}^{\times g}\,/\,\forall i,j, z_i\overline{z_i}=z_j\overline{z_j}\}$$
$$\simeq \mathbb{R}_+^\times\times \mathbb{S}^1(\mathbb{R})^g$$
where $\mathbb{S}^1(\mathbb{R})=\{z\in\mathbb{C}\,/\,z\overline{z}=1\}$. In particular, this group is connected and $P(\mathbb{R})$ also.

\end{proof}

We define $U=\exp(W_{-2}Lie(P))$. We let $X_\infty$ be the $P(\mathbb{R})U(\mathbb{C})$-orbit of $\omega\circ h_\infty$.

\begin{prop}
$U$ has dimension $g$.
\end{prop}
\begin{proof}
Considering the action of $\mathbb{G}_{m,\mathbb{R}}$ on $Lie(P)$ through $\omega\circ h_\infty\circ w$ allows us to identify $W_{-2}Lie(P)$ with 
$$\left(\begin{array}{ccc} 0 & 0 & * \\ 0 & 0 & 0 \\ 0 & 0 & 0 \end{array}\right)_{1\leq i\leq g}$$
which has dimension $g$.
\end{proof}

\begin{prop}
$X_\infty$ is connected and $(P,X_\infty)$ is a mixed Shimura datum. It is a boundary component of $(G,X)$ and every one of them is of this type.
\end{prop}
\begin{proof}
The connectedness of $X_\infty$ follows from the connectedness of $P(\mathbb{R})$. The rest results from \cite{Pin}.
\end{proof}

\begin{prop}
Let us consider the Shimura data $(P/W,X_1)=(P,X_\infty)/W$ and $(P/U,X_2)=(P,X_\infty)/U$. Then, considered as analytic varieties, $X_1$ is a point, $X_2$ an affine space of dimension $g$ and $X_\infty$ an affine space of dimension $2g$.
\end{prop}
\begin{proof}
$P/W$ is a torus so $X_1$ is a finite union of points. Because of the remark following proposition 2.9. of \cite{Pin}, $X_1$ and $X_2$ are connected. Consequently, $X_1$ is a point.

As $(P/U,X_2)$ is a unipotent extension of $(P/W,X_1)$, $X_2\rightarrow X_1$ is a $W/U(\mathbb{R})$-torsor. As $W/U$ as dimension $2g$, $X_2$ is an affine space of dimension $g$.

In the same way, $X\rightarrow X_2$ is a vector bundle which has same dimension as $U(\mathbb{C})$. As $X_2$ is contractile and $U(\mathbb{C})$ has dimension $g$ over $\mathbb{C}$, we get the result.
\end{proof}

\begin{prop}
Let $S$ be a Picard variety. Then the complement of $S$ in its Baily-Borel compactification is a finite union of points.

$S$ has a toroidal compactification which admits a stratification. Every stratum is a torus of dimension between $0$ and $g-1$ over an abelian variety of dimension $g$.

\end{prop}
\begin{proof}
Following chapter six of \cite{Pin}, we know that the boundary components of the Baily-Borel compactifications are Shimura varieties associated to data $(P/W,X_1)$. Thus they are points.

According to chapter seven of \cite{Pin} and \cite{Wil00}, we know that the strata of a toroidal compactification are of the type $(P,X_\infty)/U_\sigma$ with $U_\sigma\subset U$.

Because of \cite{Pin}, 
$$\pi_m:(P,X_\infty)/U_\sigma\rightarrow (P/U,X_2)$$ 
is the fiber product over $(P/U,X_2)$ of $\mathbb{G}_m$-torsor of dimension $\dim U-\dim U_\sigma$.

To conclude, we also know that
$$\pi_a:(P/U,X_2)\rightarrow (P/W,X_1)$$ 
is an abelian scheme, thus an abelian variety of dimension $g$. 

$S$ being of dimension $2g$, $U_\sigma$ can not be of dimension $0$.

\end{proof}

Because of \cite{Pin}, we also know that both previous compactifications admit a canonical model over $E(G,X)$.

\section{The interior motive}

The boundary motive (see \cite{Wil05}) of a smooth variety $X$ over a perfect field $k$ is defined in order to give rise to an exact triangle
$$\partial M_{gm}(X)\rightarrow M_{gm}(X)\rightarrow M_{gm}^c(X)\rightarrow \partial M_{gm}(X)[1].$$

As the strategy we have in mind can not work for the whole motive of the Kuga-Sato family $M_{gm}(A^r)$, we will use idempotents cutting this motive. More precisely, we have a functor defined in \cite{WilBF} 
$$(\partial M_{gm},M_{gm},M_{gm}^c):CHM^s(S)_Q\rightarrow DM_{gm}(E(G,X))_Q.$$
We consider an idempotent of $h(A^r/S)\in CHM^s(S)_Q$ and we call $e$ its image in $DM_{gm}(E(G,X))_Q$.

We have again an exact triangle (see \cite{WilBF})
 $$\partial M_{gm}(A^r)^e\rightarrow M_{gm}(A^r)^e\rightarrow M_{gm}^c(A^r)^e\rightarrow \partial M_{gm}(A^r)^e[1].$$
 
We send back to \cite{Bon10} and \cite{Wil08} for the definition of weight structures. In \cite{Wil08}, a strategy is developed to construct the so-called interior motive of (a direct factor of) a variety. It relies on the notion of avoidance of certain weights (see definition 1.6. and 1.10. of \cite{Wil08}). Namely, in our context, we have the following result (see theorem 4.3. of \cite{Wil08}) :

\begin{theo}
Let us suppose that $\partial M_{gm}(A^r)^e$ is without weights $0$ and $-1$. Let us fix a weight filtration avoiding those weights
$$C_{\leq -2}\rightarrow \partial M_{gm}(A^r)^e\rightarrow C_{\geq 1}\rightarrow C_{\leq -2}[1].$$
(a) The motive $M_{gm}(A^r)^e$ is without weight $-1$ and the motive $M_{gm}^c(A^r)^e$ is without weight $1$. The effective Chow motives $\text{Gr}_0(M_{gm}(A^r)^e)$ and $\text{Gr}_0(M_{gm}^c(A^r)^e)$ (see \cite{Wil08} for a precise definition) are defined and they carry an action of the endomorphism algebra of $M_{gm}(A^r)^e$ and $M_{gm}^c(A^r)^e$.\\
(b)There are canonical exact triangles 
$$C_{\leq -2}\rightarrow M_{gm}(A^r)^e\rightarrow \text{Gr}_0(M_{gm}(A^r)^e)\rightarrow C_{\leq -2}[1]$$
$$C_{\geq 1}\rightarrow \text{Gr}_0(M_{gm}^c(A^r)^e)\rightarrow M_{gm}^c(A^r)^e\rightarrow C_{\geq 1}[1]$$
stable under the actions of the endomorphisms of $M_{gm}(A^r)^e$ and $M_{gm}^c(A^r)^e$.
(c) There is a canonical isomorphism $\text{Gr}_0(M_{gm}(A^r)^e)\stackrel{\sim}{\rightarrow}\text{Gr}_0(M_{gm}^c(X)^e)$ of Chow motives which is uniquely determined by the property of factorizing the canonical morphism $M_{gm}(A^r)^e\rightarrow M_{gm}^c(A^r)^e$.\\
(d) Let $N\in CHM(E(G,X))_Q$ be a Chow motive. Then we have canonical isomorphisms
$$\Hom_{CHM(E(G,X))_Q}(\text{Gr}_0(M_{gm}(A^r)^e),N)\simeq\Hom_{DM_{gm}(E(G,X))_Q}(M_{gm}(A^r)^e,N)$$
$$\Hom_{CHM(E(G,X))_Q}(N,\text{Gr}_0(M_{gm}^c(A^r)^e))\simeq\Hom_{DM_{gm}(E(G,X))_Q}(N,M_{gm}^c(A^r)^e).$$
(e) Let $M_{gm}(A^r)^e\rightarrow N\rightarrow M_{gm}^c(A^r)^e$ be a factorization of the canonical morphism through a Chow motive. Then $\text{Gr}_0(M_{gm}(A^r)^e)= \text{Gr}_0(M_{gm}^c(A^r)^e)$ is canonically a direct factor of $N$ with a canonical direct complement.
\end{theo}

The motive $\partial M_{gm}(A^r)^e$ is called the $e$-part of the boundary motive of $A^r$. When it exists, the Chow motive $\text{Gr}_0(M_{gm}(A^r)^e)$ is called the $e$-part of the interior motive of $A^r$.
\newline

We will use some specific results on a particular class of motives relevant in our context.

\begin{de}
Let $k$ be a field and $\overline{k}$ an algebraic closure of $k$.\\
(a) The category of abelian Chow motives over $\overline{k}$ is the strict full dense additive tensor sub-category of $CHM(\overline{k})_Q$ generated by Chow motives of abelian variety over $\overline{k}$ and by the motives $\mathbb{Z}(m)[2m]$.\\
(b) The category of abelian Chow motives over $k$ is the strict full dense additive tensor sub-category of $CHM(k)_Q$ generated by Chow motives whose base change to $\overline{k}$ is an abelian motive over $\overline{k}$.\\
(c) The triangulated category of abelian motives $DM^{Ab}(k)_Q$ is the strict full triangulated sub-category of $DM_{gm}(k)_Q$ generated by abelian Chow motives.
\end{de}

\begin{prop}
$\partial M_{gm}(A^r)$ is an abelian motive in $DM_{gm}(E(G,X))_Q$
\end{prop}
\begin{proof}
$A^r$ is the Shimura variety associated to the unipotent extension of $(G,X)$ by $V^r$.

We consider a toroidal compactification of $A^r$ denoted $\tilde{A^r}$. It is a Shimura variety associated to a complete admissible cone decomposition associated to the unipotent extension of $(G,X)$ by $V^r$.

This compactification admits a stratification whose unique open stratum is $j:A^r\rightarrow \tilde{A^r}$. We denote by $i_\sigma: \tilde{\mathcal{A}^r_\sigma}\rightarrow \tilde{\mathcal{A}^r}$ the non generic strata. These are Shimura varieties associated to boundary components $(\tilde{P},\tilde{X})$ of the unipotent extension of $(G,X)$ by $V^r$.

Because of the co-localization exact sequence of \cite{Wil05}, $\partial M_{gm}(\mathcal{A}^r)$ is a succesive extension of objects $M_{gm}(\tilde{\mathcal{A}^r_\sigma}, i_\sigma^!j_!\mathbb{Z})$.

In \cite{Wil05}, it is shown that we have an isomorphism
$$M_{gm}(\tilde{\mathcal{A}^r_\sigma}, i_\sigma^!j_!\mathbb{Z})\simeq M_{gm}(S^{\tilde{K}}(\tilde{P},\tilde{X}))$$
where $S^{\tilde{K}}(\tilde{P},\tilde{X})$ is the Shimura variety associated to $(\tilde{P},\tilde{X})$ and $\tilde{K}$ is a compact subgroup of $\tilde{P}(\mathbb{A}_f)$. Thus, it suffices to show that the motive of $S^{\tilde{K}}(\tilde{P},\tilde{X})$ is an abelian motive. 

$(\tilde{P},\tilde{X})$ is a  unipotent extension of a boundary component $(P,X_\infty)$ of $(G,X)$.
Writing $\tilde{U}$ for the part of weight $-2$ of the unipotent radical $\tilde{W}$ of $\tilde{P}$, we have the following devissage
$$(\tilde{P},\tilde{X})\rightarrow (\tilde{P},\tilde{X})/U\rightarrow (\tilde{P},\tilde{X})/\tilde{W}=(P/W,X_1).$$

We know that the connected components of the Shimura variety associated to $(P/W,X_1)$ are points.

Because of \cite{Pin}, we know that the Shimura variety associated to $(\tilde{P},\tilde{X})/U$ is an abelian variety. Thus, its motive is an abelian motive.

Still referring to the results of \cite{Pin}, we know that the Shimura variety associated to $(\tilde{P},\tilde{X})$ is a torus-torsor over the Shimura variety associated to $(\tilde{P},\tilde{X})/U$. The base change to an algebraic closure of such a motive decomposes into a motive whose direct factors are Tate twists and shifts of the motive associated to $(\tilde{P},\tilde{X})/U$.

Therefore, it is also abelian.

\end{proof}

When $k$ is a field embeddable into $\mathbb{C}$ (by a morphism $\tau$), we have the Hodge theoretic realization
$$R_\tau:DM_{gm}(k)_Q\rightarrow D$$
where $D$ is the bounded derived category of mixed graded-polarizable $\mathbb{Q}$-Hodge structure tensored with $Q$ (cf. \cite{P-St}).

For a prime $l$ no dividing the characteristic of $k$, we have the $l$-adic realization 
$$R_l:DM_{gm}(k)_Q\rightarrow D$$
with here $D$ the bounded derived category of constructible $\mathbb{Q}_l$-sheaves on $\text{Spec}(k)$ tensored with $Q$ (cf. \cite{E}).

We write $H^*$ for the cohomological functors associated to the t-structure of $D$.

For abelian motives,  the avoidance of certain weights can be checked on the realizations. More precisely, we have (see theorem 1.13. of \cite{Wil14}) :

\begin{theo}
Let $k$ be a field embeddable into $\mathbb{C}$ and $Q$ a finite direct product of fields of characteristic zero. Let $M$ be an object of $DM_{gm}(k)_Q$, $\alpha\leq\beta$ two integers with $\alpha\leq -1$ and $\beta\geq 0$. Assume that $M$ is an abelian motive. 
Then $M$ is without weights $\alpha,\cdots,\beta$ if and only if its cohomology objects $H^nR(M)$ are without weights $n-\beta,\cdots,n-\alpha$. 
\end{theo}

We go back to our Picard variety $S$. We recall the existence of a functor 
$$\mu:\Rep(G_Q)\rightarrow VHS(S)_Q$$
where $VHS(S)_Q$ is the category of admissible graded-polarizable variations of $Q$-Hodge structure on $S$.

As the boundary motive of a variety can not avoid weights $0$ and $-1$ without being trivial, we need to find idempotent. For the case of the Kuga-Sato family $A^r$, we can use the following result of \cite{Anc1}.

\begin{theo}
There is a $Q$-linear tensor functor which commutes to Tate twists
$$\tilde{\mu}:\Rep(G_Q)\rightarrow CHM^s(S)_Q$$
such that its composition with the Hodge realization is isomorphic to $\mu$ and such that the image of the trivial representation is the unit and the image of $V^\vee$ is $h^1(A/S)$.
\end{theo}

Combined with the functor $(\partial M_{gm},M_{gm},M_{gm}^c):CHM^s(S)_Q\rightarrow DM_{gm}(E(G,X))_Q$ of \cite{WilBF}, we get idempotents of $M_{gm}(A^r)$ from idempotents of the $G_\mathbb{C}$-representation $\Lambda^p(V_\mathbb{C}^{\vee\oplus r})$.

Consequently, we can study specific representations of $G$. First, we need to fix some data such as Borel subgroups and maximal tori.

We choose as a maximal torus of $G_\mathbb{C}$ the one called $T_m$ in proposition 2.3. and the cartesian product of the upper triangular matrices as a Borel subgroup. The characters of $T_m$ will be written under the form $\lambda=((a_i,b_i,c_i)_{1\leq i\leq g},d)$.

We call $\lambda_{i,j}$, with $1\leq i\leq 3$ and $1\leq j\leq g$, the elements of the canonical basis of the characters of $T_m$. $\Psi$ will be the root system associated to the previous data and $\Psi^+$ the positive roots. The elements of $\Psi^+$ are
$$e_{12}^j=\lambda_{1,j}\lambda_{2,j}^{-1}$$ 
$$e_{23}^j=\lambda_{2,j}\lambda_{3,j}^{-1}$$  
$$e_{13}^j=\lambda_{1,j}\lambda_{3,j}^{-1}$$ 
with $1\leq j\leq g$. 

The dominant characters are those verifying $a_i\geq b_i\geq c_i$, amongst which the regular one are those with $a_i>b_i>c_i$.

For a dominant character $\lambda$, $V_\lambda$ will be the irreducible representation of highest weight $\lambda$.

For example, we have the flowing decomposition
$$V_\mathbb{C}=\bigoplus\limits_{i=1}^gV_{i,+}\times V_{i,-}.$$
$V_{i,+}$ corresponds to the character $((a_j,b_j,c_j)_{1\leq j\leq g},d)$ with $a_j=b_j=c_j=0$ for all $j\neq i$, $a_i=1$, $b_i=c_i=0$ and $d=0$. $V_{i,-}$ corresponds to $((a_j,b_j,c_j)_{1\leq j\leq g},d)$ with $a_j=b_j=c_j=0$ for all $j\neq i$, $a_i=b_i=0$, $c_i=-1$ and $d=1$.

We also have
$$V_\mathbb{C}^\vee=\bigoplus\limits_{i=1}^gV_{i,+}^\vee\times V_{i,-}^\vee.$$
$V_{i,+}^\vee$ corresponds to $((a_j,b_j,c_j)_{1\leq j\leq g},d)$ with $a_j=b_j=c_j=0$ for all $j\neq i$, $a_i=b_i=0$, $c_i=-1$ and $d=0$. $V_{i,-}$ corresponds to $((a_j,b_j,c_j)_{1\leq j\leq g},d)$ with $a_j=b_j=c_j=0$ for all $j\neq i$, $a_i=1$, $b_i=c_i=0$ and $d=-1$.

\begin{prop}
The irreducible sub-$G_\mathbb{C}$-representations of $\Lambda^p((V_\mathbb{C}^\vee)^{\oplus r})$ are those $V_{((a_i,b_i,c_i),d)}$ such that :
$$p=-2d-\sum\limits_{i=1}^g(a_i +b_i+c_i),$$
$$-r\leq c_i\leq b_i\leq a_i\leq r,$$
$$3gr+\sum\limits_{i=1}^g(c_i^-+b_i^-+a_i^-)\geq -d\geq \sum\limits_{i=1}^g(c_i^++b_i^++a_i^+).$$
The two last lines have to be true for every $1\leq i\leq g$. We use the notations $x_+=max(x,0)$ et $x_-=min(x,0)$.
\end{prop}
\begin{proof}
The previous decomposition of $V_\mathbb{C}^\vee$ gives a $\mathbb{Q}$-basis $\mathcal{B}$ of $g(3+3)$ vectors which are eigenvectors for the $T_m$-action; thus it gives a $\mathbb{Q}$-basis $\mathcal{B}_p$ of $\Lambda^p((V_\mathbb{C}^\vee)^{\oplus r})$ composed of exterior products of vectors taken from $\mathcal{B}$ which in turn is made of eigenvectors for the $T_m$-action. 

Giving a vector of this basis $\mathcal{B}_p$ is equivalent to giving $p$ vectors of $\mathcal{B}$. For such a vector, let us denote $\alpha_{i,+}$ (resp. $\beta_{i,+}$, resp. $\gamma_{i,+}$) the number of vectors of $\mathcal{B}$ chosen as the first (resp. the second, resp. the third) vector of the canonical basis of $V_{(\cdots (0, 0, 0), (0, 0, -1), (0, 0, 0), \cdots, 0)}$ and let us denote $\alpha_{i,-}$ (resp. $\beta_{i,-}$, resp. $\gamma_{i,-}$) the number of vectors of $\mathcal{B}$ chosen as the first (resp. the second, resp. the third) vector of the canonical basis of $V_{(\cdots (0, 0, 0), (1, 0, 0), (0, 0, 0), \cdots, -1)}$. Then we have $\alpha_{i,+/-},\beta_{i,+/-},\gamma_{i,+/-}\in[0,r]$. 

$G_\mathbb{C}$ acts on such a vector via the character $((a_i,b_i,c_i),d)$ such that 
$$a_i=\alpha_{i,-}-\alpha_{i,+}\;,\;b_i=\beta_{i,-}-\beta_{i,+}\;,\;c_i=\gamma_{i,-}-\gamma_{i,+}\;,$$ 
$$d=-\sum\limits_{i=1}^g(\alpha_{i,-}+\beta_{i,-}+\gamma_{i,-}).$$

Now, giving an irreducible sub-representation of $\Lambda^p(V_\mathbb{C}^\vee)^{\oplus r}$ is the same thing as giving a maximal weight in this space, i.e. giving an eigenvector whose weight is maximal .

Therefore we find the inequalities written in the statement of the proposition, $a_i\geq b_i\geq c_i$ coming from the choice of the Borel subgroup.\newline

For the reverse, let us consider 
$$W_{a_i,b_i,c_i}^+=(\Lambda^3 V_{i,-}^\vee)^{\otimes c_{+,i}}\otimes(\Lambda^2 V_{i,-}^\vee)^{\otimes (b_{i,+}-c_{i,+})}\otimes (\Lambda V_{i,-}^\vee)^{\otimes (a_{i,+}-b_{i,+})}$$
$$W_{a_i,b_i,c_i}^-=(\Lambda^3 V_{i,+}^\vee)^{\otimes (-a_{i,-})}\otimes(\Lambda^2 V_{i,+}^\vee)^{\otimes (a_{i,-}-b_{i,-})} \otimes(\Lambda V_{i,+})^{\otimes (b_{i,-}-c_{i,-})}$$
$$\mathbb{W}=(\bigotimes\limits_{i=1}^g W_{a_i,b_i,c_i}^+\otimes W_{a_i,b_i,c_i}^-)\otimes V_{0,\cdots,0,-1}^{\otimes (-d-S_+)}$$
where $S_+=\sum\limits_{i=1}^g a_{i,+}+b_{i,+}+c_{i,+}$. This is allowed because $-d\geq S_+$ fby hypothesis.

We easily check that $\mathbb{W}$ is a representation of highest weight $((a_i,b_i,c_i),d)$. To check that it is a sub-representation of $\Lambda^p((V_\mathbb{C}^\vee)^{\oplus r})$ it suffices to show that $V_{0,\cdots ,0,-1}$ is a sub-representation of $V_{i,+}^\vee\otimes V_{i,-}^\vee$.

The pairing $J$ on $V_{i,+}\otimes V_{i,-}$ induces a surjective morphism of $G_\mathbb{C}$-representations $V_{i,+}\otimes V_{i,-}\rightarrow V_{0,\cdots,0,1}$, which permits to see by duality $V_{0,\cdots,0,-1}$ as a sub-representation of $V_{i,+}^\vee\otimes V_{i,-}^\vee$.
\end{proof}

We let $e_\lambda$ be an idempotent of $\Lambda^p(V_\mathbb{C}^{\vee \oplus r})$ corresponding to $V_\lambda$. We denote by the same symbol $e_\lambda$ the idempotent of $M_{gm}(A^r)$. We need to compute boundary cohomology of $M_{gm}(A^r)^{e_\lambda}$.

\begin{prop}
Let $Q$ be a field containing $E_{gal}$ and $\lambda=((a_i,b_i,c_i),d)$ a character associated to an irreducible sub-representation of $\Lambda^p(V_Q^{\vee\oplus r})$. Then there exists an canonical isomorphism 
$$(\partial H^n(A^r(\mathbb{C}),\mathbb{Q})\otimes_\mathbb{Q}Q)^{e_\lambda}\simeq\partial H^{n-p}(S(\mathbb{C}),\mu(V_\lambda)).$$ 
\end{prop}
\begin{proof}
We denote by $e_\lambda$ the idempotent defining $V_\lambda$ as a direct factor of $\Lambda^*(V_Q^{\vee\oplus r})$. Applying the functor $\mu$ we get an idempotent acting on the relative cohomology $\mathcal{H}^q(\mathcal{A}^r/S)$ for every $q$. Moreover, we have 
$$\mu(\Lambda^p(V_Q^{\vee\oplus r}))=\mathcal{H}^p(\mathcal{A}^r/S).$$
The image of this idempotent is $\mu(V_\lambda)$ if $q=p$ and $0$ otherwise. We also have the spectral sequence of Hodge structures
$$E_2^{m,n}=\partial H^m(S(\mathbb{C}),\mathcal{H}^n(\mathcal{A}^r/S))\Rightarrow \partial H^{m+n}(\mathcal{A}^r(\mathbb{C}),\mathbb{Q})$$
which is compatible with the action of the Chow group 
$$CH^{3rg}(\mathcal{A}^r\times_S \mathcal{A}^r)\otimes_\mathbb{Z} Q=End_{CHM^s(S)_Q}(h(\mathcal{A}^r/S)).$$
Taking the part fixed by $e_\lambda$, we get
$$\partial H^m(S(\mathbb{C}),\mathcal{H}^n(\mathcal{A}^r/S)^{e_\lambda})\Rightarrow (\partial H^{m+n}(\mathcal{A}^r(\mathbb{C}),\mathbb{Q})\otimes Q)^{e_\lambda}$$
which gives 
$$\partial H^m(S(\mathbb{C}),\mu(V_\lambda))\simeq(\partial H^{m+p}(\mathcal{A}^r(\mathbb{C}),\mathbb{Q})\otimes Q)^{e_\lambda}.$$
\end{proof}

Thus we are reduced to calculate the weights appearing in $R^ki^*j_*\mu(V_\lambda)$, where $j:S\rightarrow \overline{S}$ is the inclusion of our Picard variety into its Baily-Borel compactification and $i:\partial S\rightarrow \overline{S}$ is the inclusion of the boundary.

For that, we are going to used the main result of \cite{BW}. Following the notations of this article, we denote by $\overline{H_C}$ the image in $Q/W$ of $\text{Cent}_{Q(\mathbb{Q})}(X_1)\cap P(\mathbb{A}_f).K$.

\begin{theo}
For every $\mathbb{V}^.\in D^b(Rep_G)$ we have a canonical and functorial isomorphism
$$\mathcal{H}^n i^*j_*\circ\mu_S(\mathbb{V}^.)\simeq\bigoplus_{p+q=n}\mu_{pt}\circ H^p(\overline{H_C},H^q(W,Res_Q^G\mathbb{V}^.))$$
of variations of Hodge structure, where $pt$ is a connected component of the boundary of the Baily-Borel compactification of $S$, $i:pt\rightarrow S^*$ the corresponding inclusion and $j:S\rightarrow S^*$ the inclusion of the Picard variety in its Baily-Borel compactification. 
\end{theo}

In order to calculate the cohomology of the unipotent radical, we use Kostant's formula (see \cite{War}). 

\begin{theo}
Let $G$ be a reductive group over $\mathbb{C}$, $B$ a Borel sub-group, $W$ its unipotent radical, $\rho$ the half sum of positive roots, $\mathcal{R}$ the Weyl group and $V_\lambda$ the irreducible $G$-representation of highest weight $\lambda$. For all $\sigma\in\mathcal{R}$, we note $l(\sigma)$ the length of $\sigma$.

Then we have the following equality of $B/W$-representations
$$H^k(W,V_{\lambda |B})=\bigoplus\limits_{l(\sigma)=k}V_{\sigma(\lambda+\rho)-\rho,B/W}$$
where $V_{\alpha,B/W}$ is the irreductible representation of $B/W$ of highest weight $\alpha$.
\end{theo}

Now, we describe our data in the language used in the previous theorems.\newline

The Weyl group is $\mathfrak{S}_3^g$. We recall that the length of a element of an element $\sigma$ of the Weyl group is
$$l(\sigma)=\text{card}\{\alpha\in\Psi^+\,/\,\sigma^{-1}\alpha\notin\Psi^+\}.$$

More precisely, if $(\sigma_i)_{1\leq i\leq g}\in\mathfrak{S}_3^g$, we can see each $\sigma_i$ as an element of the Weyl group associated to the root system of $GL_{3,\mathbb{C}}$ (for the upper triangular matrices and the diagonal torus). Then, 
$$l((\sigma_i)_{1\leq i\leq g})=\sum\limits_{i=1}^gl(\sigma_i).$$

Let $\rho$ be the half-sum of the positive roots. Then $\rho=((\rho_i)_{1\leq i\leq g},0)$ with $\rho_i=(1,0,-1)$. Moreover, 

$$l(\text{id})=0$$
$$l((12))=l((23))=1$$
$$l((123))=l((132))=2$$
$$l((13))=3.$$

For a root $\lambda_i$ of $GL_{3,\mathbb{C}}$ we calculate each of the values of $\sigma_i(\lambda_i+\rho_i)-\rho_i$ :
$$\text{id}(\lambda_i+\rho_i)-\rho_i=(a_i,b_i,c_i)$$
$$(12)(\lambda_i+\rho_i)-\rho_i=(b_i-1,a_i+1,c_i)$$
$$(23)(\lambda_i+\rho_i)-\rho_i=(a_i,c_i-1,b_i+1)$$
$$(123)(\lambda_i+\rho_i)-\rho_i=(c_i-2,a_i+1,b_i+1)$$
$$(132)(\lambda_i+\rho_i)-\rho_i=(b_i-1,c_i-1,a_i+2)$$
$$(13)(\lambda_i+\rho_i)-\rho_i=(c_i-2,b_i,a_i+2).$$

Moreover, $Q/W$ is a torus. As a consequence, its irreducible complex representations are of dimension one.\newline

If $K$ a congruence subgroup of level $N$, then
$$\overline{H_C}=\{(\left(\begin{array}{ccc} z & 0 & 0 \\ 0 & 1 & 0 \\ 0 & 0 & \alpha(z)^{-1} \end{array}\right), 1)\,/\,z\in U_{E,N}\}$$
where $U_{E,N}=\{z\in\mathcal{O}_E\,/\, z=1\mod N\mathcal{O}_E\}$. 

More generally, if $K$ is an open compact neat subgroup of $G(\mathbb{A}_f)$, $\overline{H_C}$ is a finite index subgroup of $\mathcal{O}_E^\times$ without torsion. Hence, 
$$\overline{H_C}\simeq\mathbb{Z}^{g-1}.$$

\begin{prop}
Let $M$ be a one dimensional $\mathbb{C}$-vector space with an action of $\overline{H_C}$.

If $\overline{H_C}$ acts non trivially, 
$$\forall k\geq 0, H^k(\overline{H_C},M)=0.$$

If this action is trivial,
$$H^k(\overline{H_C},M)=M^{\left(\begin{array}{c} g-1 \\ k \end{array}\right)}.$$
\end{prop}
\begin{proof}
When the action is trivial, we apply Hochschild-Serre spectral sequence to a subgroup of $\overline{H_C}$ isomorphic to $\mathbb{Z}^{g-2}$ to get an exact sequence
$$0\rightarrow H^0(\mathbb{Z},H^q(\mathbb{Z}^{g-2},M))\rightarrow H^q(\mathbb{Z}^{g-1},M)\rightarrow H^1(\mathbb{Z},H^{q-1}(\mathbb{Z}^{g-2},M))\rightarrow 0 .$$
Having in mind Pascal's triangle relation, we are reduced to show the case where $g-1=1$ which is trivial.\newline

When the action is not trivial, we consider the same spectral sequence applied to a subgroup of $\overline{H_C}$ isomorphic to $\mathbb{Z}^{g-2}$ which acts non trivially on $M$. Again, it suffices to treat the case $g-1=1$. 

On the one hand, as $M$ has dimension 1 and the action of $\mathbb{Z}$ on it is non trivial, 
$$H^0(\mathbb{Z},M)=0.$$ 

On the other hand, a cocycle of $Z^1(\mathbb{Z},M)$ is determined by the image of $1\in\mathbb{Z}$. Thus we have $Z^1(\mathbb{Z},M)\simeq M$. With this identification, the coboundaries form the following set
$$B^1(\mathbb{Z},M)=\{1.a-a,a\in M\}.$$ 

As the action of $\mathbb{Z}$ on $M$ is non trivial $B^1(\mathbb{Z},M)=M$. Hence
$$H^1(\mathbb{Z},M)=0.$$
\end{proof}

Let $V_{((a_,b_i,c_i)_{1\leq i\leq g},d)}$ be an irreducible representation of $T_m$. Seeing $\overline{H_C}$ as a subgroup of $\mathcal{O}_E^\times$, we are allowed to write that $z\in\overline{H_C}$ acts on $V_{((a_,b_i,c_i)_{1\leq i\leq g},d)}$ by multiplication by $\prod\limits_{I=1}^g \sigma_i(z)^{a_i}\overline{\sigma_i(z)}^{-c_i}$.

\begin{prop}
The action of $\overline{H_C}$ on $V_{((a_i,b_i,c_i),d)}$ is trivial if and only if there exists $m\in\mathbb{Z}$ such that for every $1\leq i\leq g$ 
$$a_i-c_i=m$$
\end{prop}
\begin{proof}
First, let us suppose that $\overline{H_C}$ acts trivially on $V_{((a_i,b_i,c_i),d)}$, meaning that for every $\epsilon\in\overline{H_C}$, $\prod\limits_{i=1}^g\sigma_i(\epsilon)^{a_i}\overline{\sigma_i(\epsilon)}^{-c_i}=1$. Then for every $\epsilon\in\overline{H_C}$,
$$\sum\limits_{i=1}^g(a_i-c_i)\log(|\sigma_i(\epsilon)|)=0.$$
Thus the vector $(a_i-c_i)_{1\leq i\leq g}$ is orthogonal to the vectors $(\log(|\sigma_i(\epsilon)|))_{1\leq i\leq g}$ which generate a space of dimension $g-1$ according to Dirichlet's unit theorem. The orthogonal to this space is the line spanned by $(1,\cdots,1)$. Consequently, there exists $m\in\mathbb{Z}$ such that $1\leq i\leq g$,
$$a_i-c_i=m.$$
On the other hand, let us suppose there exists $m\in\mathbb{Z}$ such that $a_i-c_i=m$ for every $1\leq i\leq g$; then for every $\epsilon\in\overline{H_C}$,
$$\prod\limits_{i=1}^g|\sigma_i(\epsilon)^{a_i}\overline{\sigma_i(\epsilon)}^{-c_i}|^2=|\prod\limits_{i=1}^g\sigma_i(\epsilon)\overline{\sigma_i(\epsilon)}|^m=1.$$
As $\overline{H_C}$ is neat and the image of a neat group is neat, we have 
$$\prod\limits_{i=1}^g\sigma_i(\epsilon)^{a_i}\overline{\sigma_i(\epsilon)}^{-c_i}=1.$$
\end{proof}

Linked to the notion of positivity of certain weights, we will say that $\sigma_i\in\mathfrak{S}_3$ is positive if $l(\sigma_i)\leq 1$ and negative otherwise. $(\sigma_i)$ will be called totally positive (resp. negative) if for every $1\leq i\leq g$, $\sigma_i$ is positive (resp. negative).

\begin{prop}
If $\sigma=(\sigma_i)_{1\leq i\leq g}$ is not totally positive or totally negative, then for every irreducible sub-representation $V_\lambda$ of $\Lambda^p(V_\mathbb{C}^{\oplus r})$, $\overline{H_C}$ acts non trivially on $V_{\sigma(\lambda+\rho)-\rho}$.
\end{prop}
\begin{proof}
Let us suppose for example that $\sigma_i=(12)$ and $\sigma_j=(123)$. Writing $\lambda=((a_i,b_i,c_i)_{1\leq i\leq g},d)$, we get 
$$\sigma_i((a_i,b_i,c_i)+\rho_i)-\rho_i=(b_i-1,a_i+1,c_i)$$
$$\sigma_j((a_j,b_j,c_j)+\rho_j)-\rho_j=(c_j-2,a_j+1,b_j+1).$$
For the action of $\overline{H_C}$ to be trivial, we must have $b_i-1-c_i=c_j-2-b_j-1=m$. As $b_i\geq c_i$, $m\geq -1$. As $b_j\geq c_j$, $m\leq -3$.

Thus, it is impossible. The other cases are treated the same way.
\end{proof}

We are now able to calculate the weights appearing in $\partial H^n(S(\mathbb{C}),\mu(V_\lambda))$ for an irreducible direct factor $V_\lambda$ of $\Lambda^p(V_\mathbb{C}^{\vee\oplus r})$. \newline

Let $(p_0,q_0)$ be a couple of non negative integers such that $n=p_0+q_0$. Theorem 3.4. brings us to study weights appearing in $H^{p_0}(\overline{H_C}, H^{q_0}(W,Res_Q^G V_\lambda))$.

Due to theorem 3.5., we take $\sigma\in\mathfrak{S}_3^g$ such that $l(\sigma)=q_0$. Propositions 3.4., 3.5. and 3.6. allow us to suppose that $\sigma$ is totally positive or totally negative.

If $\sigma$ is totally positive, let us denote 
$$I=\{i/1\leq i\leq g\,\text{ and }\, \sigma_i=id\}$$
$$J=\{j/1\leq j\leq g\,\text{ and }\, \sigma_j=(12)\}$$
$$K=\{k/1\leq k\leq g\,\text{ and }\, \sigma_k=(23)\}.$$
Then we have $l(\sigma)=|J|+|K|$. Moreover, propositions 3.4. and 3.5. allow us to consider $\lambda=((a_i,b_i,c_i)_{1\leq i\leq g},d)$ such that there exists $m$ such that 
$$\forall i,j,k\in I\times J\times K,\; a_i-c_i=b_j-1-c_j=a_k-b_k-1=m.$$
As $a_i\geq c_i$, $b_j\geq c_j$ and $a_k\geq b_k$, we must have 
$$m\geq -1.$$ 
For a regular character, 
$$m\geq 0.$$
Under these conditions, using the definition of the morphism $\omega\circ h_\infty$ describing the Shimura datum of the boundary $\partial S$, we see that $V_{\sigma(\lambda+\rho)-\rho}$ will be of weight 
$$w=-2d-\sum\limits_i (2a_i+b_i)-\sum\limits_j(2b_j+a_j-1)-\sum\limits_k(2a_k+c_k-1)$$
$$=-2d-\sum\limits_{l=1}^g(a_l+b_l+c_l)+\sum\limits_i(c_i-a_i)+\sum\limits_j(c_j-b_j+1)+\sum\limits_k(b_k-a_k+1)$$
$$=p-mg.$$
Besides, as $\overline{H_C}$ has cohomological dimension $g-1$ and $0\leq |J|+|K|\leq g$, we suppose that
 $$0\leq p_0\leq g-1$$
 $$0\leq q_0\leq g$$
 $$0\leq n\leq 2g-1.$$
 
Analogously, for $\sigma$ totally negative, we write
$$I=\{i/1\leq i\leq g\,\text{ et }\, \sigma_i=(123)\}$$
$$J=\{j/1\leq j\leq g\,\text{ et }\, \sigma_j=(132)\}$$
$$K=\{k/1\leq k\leq g\,\text{ et }\, \sigma_k=(13)\}.$$
We have $l(\sigma)=2|I|+2|J|+3|K|$. Propositions 3.4. and 3.5. allow us to consider elements $\lambda=((a_i,b_i,c_i)_{1\leq i\leq g},d)$ such that there exists $m$ such that
$$\forall i,j,k, \; c_i-2-b_i-1=b_j-1-a_j-2=c_k-2-a_k-2=-m.$$
As $b_i\geq c_i$, $a_j\geq b_j$ and $a_k\geq c_k$, we have 
$$m\geq 3.$$
For a regular character, 
$$m\geq 4.$$ 
$V_{\sigma(\lambda+\rho)-\rho}$ will be of weight
$$w=-2d-\sum\limits_i(2c_i+a_i-3)-\sum\limits_j(2b_j+c_j-3)-\sum\limits_k(2c_k+b_k-4)$$$$
=-2d-\sum\limits_{l=1}^g(a_l+b_l+c_l)+\sum\limits_i(b_i-c_i+3)+\sum\limits_j(a_j-b_j+3)+\sum\limits_k(a_k-c_k+4)$$
$$=p+mg.$$
This time, we have 
$$0\leq p_0\leq g-1$$
$$2g\leq q_0\leq 3g$$
$$2g\leq n\leq 4g-1.$$

\begin{theo}
If $\lambda$ is a regular character of $G_\mathbb{C}$, the $e_\lambda$-part of the boundary motive of $A^r$ is without weight $0$ et $-1$.
\end{theo}
\begin{proof}
Due to theorem 3.2. and proposition 3.1., it suffices to check that $(\partial H^{n+p}(\mathcal{A}^r,\mathbb{Q})\otimes\mathbb{C})^{e_\lambda}\simeq\partial H^n(S(\mathbb{C}),\mu(V_\lambda))$ is without weights $n+p$ and $n+p+1$.\newline

If $n\leq 2g-1$, because of the previous calculation, this Hodge structure is either zero either of weight $p-mg$ that we write $n+p-n-mg$ with $\lambda=((a_i,b_i,c_i)_{1\leq i\leq g},d)$ such that there exists a decomposition $[\![1,g]\!]=I\cup J\cup K$ and an integer $m$ such that
$$\forall i\in I, j\in J, k\in K,\; m=a_i-c_i=b_j-1-c_j=a_k-b_k-1.$$
We saw that in this case $m\geq 0$. Thus $-n-mg<0$ unless $n=0$ and $m=0$. When $n=0$, we must have $I=[\![1,g]\!]$, meaning that $a_i-c_i=m$ for every $1\leq i\leq g$. As $\lambda$ is regular, we have for all $1\leq i\leq g$
$$a_i>c_i.$$
So $m=a_i-c_i>0$. Thus $-n-mg<0$ also in this case.\newline

If $n\geq 2g$, the Hodge structure considered is either zero either of weight $p+mg=n+p-n+mg$ with $\lambda=((a_i,b_i,c_i)_{1\leq i\leq g},d)$ such that there exists a decomposition $[\![1,g]\!]=I\cup J\cup K$ and an integer $m$ such that
$$\forall i\in I, j\in J, k\in K,\; -m=c_i-2-b_i-1=b_j-1-a_j-2=c_k-2-a_k-2.$$
We saw in this case that $m\geq 4$. Thus $-n+mg>1$ unless $m=4$ and $n=4g-1$. If $K\neq \emptyset$, $m> 4$. Thus, if $m=4$, $K=\emptyset$. Then, with the previous notations, we have
$$n=p_0+q_0$$
$$q_0=2g$$
$$0\leq p_0\leq g-1.$$
Consequently, $n<3g$, showing that the prohibited weights are avoided as wanted. 
\end{proof}

\begin{coro}
Let $\lambda$ be a regular character of $G_\mathbb{C}$. Then the $e_\lambda$-part of the interior motive of $A^r$ is well defined.
\end{coro}

When the boundary motive is trivial, avoidance of weights is trivially fulfilled. In case of Hilbert-Blumenthal varieties (see \cite{Wil09}), the condition for the boundary motives considered to be trivial concerns the character defining the motive; it is a condition of parallelism.

In the case of Picard varieties, we can find a condition for the $e_\lambda$-part of the boundary motive of $A^r$ to be trivial. 

\begin{de}
We call a character $\lambda=((a_i,b_i,c_i),d)$ of $\mathbb{G}_{m,\mathbb{C}}^{4g-1}$ Kostant-parallel if there exists an integer $m$ and a decomposition $[\![1,g]\!]=I\cup J\cup K$ with
$$I=\{i\in[\![1,g]\!]\,/\,a_i-c_i=m\}$$
$$J=\{j\in[\![1,g]\!]\,/\,b_j-c_j-1=m\}$$
$$K=\{k\in[\![1,g]\!]\,/\,a_k-b_k-1=m\}.$$
\end{de}

\begin{prop}
The $e_\lambda$-part of the boundary motive of $A^r$ is trivial if and only if $\lambda$ is not Kostant-parallel.
\end{prop}
\begin{proof}
Because of proposition 3.5., $\overline{H_C}$ acts trivially on the groups $H^q(W,Res_Q^G V_\lambda)$ if and only if $\lambda$ is not Kostant-parallel. 

Because of proposition 3.3. and theorem 3.4., Hodge realization of the $e_\lambda$-part of the boundary motive of $A^r$ is trivial if and only if $\lambda$ is not Kostant-parallèle. 

Because of the conservativity of the realizations on abelian motives (see theorem 1.12. of \cite{Wil14}), the $e_\lambda$-part of the boundary motive of $A^r$ is trivial if and only if $\lambda$ is not Kostant-parallel.
\end{proof}

In the case of Picard surfaces (see \cite{Wil14}), the regularity of the character $\lambda$ is equivalent to the avoidance of weights $-1$ and $0$ by the $e_\lambda$-part of the boundary motive. The next proposition shows that in the case of Picard varieties there exists non regular characters $\lambda$ such that weights $0$ et $-1$ are avoided without having the $e_\lambda$-part of the boundary motive of $A^r$ being trivial.

\begin{prop}
Let us suppose that $g=2$. We consider the character $((1,0,-1),(0,0,0),-1)$ whose is Kostant-parallel, dominant but not regular. The representation of highest weight associated is a direct factor of $\Lambda^2(V_\mathbb{C}^\vee)$. Neverthelesse, the $e_\lambda$-part of the boundary motive of $A$ is without weight $0$ and $-1$.
\end{prop}
\begin{proof}
The previous calculations show that 
$$(\partial H^{0}(S(\mathbb{C}),\mu(V_\lambda))=0$$ 
$$(\partial H^{3}(S(\mathbb{C}),\mu(V_\lambda))=0$$
$$(\partial H^{4}(S(\mathbb{C}),\mu(V_\lambda))=0$$
$$(\partial H^{7}(S(\mathbb{C}),\mu(V_\lambda))=0$$
$$(\partial H^{1}(S(\mathbb{C}),\mu(V_\lambda))\text{ has weight }2=1+2-1$$ 
$$(\partial H^{2}(S(\mathbb{C}),\mu(V_\lambda))\text{ has weight }2=2+2-2$$
$$(\partial H^{5}(S(\mathbb{C}),\mu(V_\lambda))\text{ has weight }2+4\times2=5+2+3$$ 
$$(\partial H^{6}(S(\mathbb{C}),\mu(V_\lambda))\text{ has weight }2+4\times2=6+2+2.$$
This shows that the $e_\lambda$-part of the boundary motive of $A$ is without weight $-1$ and $0$.
\end{proof}

However we can find some type of characters amongst which the avoidance of weights $0$ et $-1$ is equivalent to regularity.

\begin{de}
We say that a character $((a_i,b_i,c_i)_{1\leq i\leq g},d)$ of $T_m\subset G_\mathbb{C}$ is parallel if for every $1\leq i,j\leq g$,
$$a_i=a_j$$
$$b_i=b_j$$
$$c_i=c_j$$
\end{de}

\begin{prop}
If $\lambda$ is a dominant parallel character of $T_m\subset G_\mathbb{C}$, then $M_{gm}(A^r)^{e_\lambda}$ is without weights $0$ et $-1$ if and only if $\lambda$ is regular.
\end{prop}
\begin{proof}
We have already shown that a regular character avoids weights $0$ et $-1$. 

For the reverse, let us suppose that $\lambda$ is a non regular parallel character. Let us suppose for example that there exists $1\leq i\leq g$ such that $a_i=b_i$ (the case $b_i=c_i$ being similar). Then for every $1\leq j\leq g$, $a_j=b_j$. Let us consider the element of the Weyl group $\tau=((12))_{1\leq j\leq g})\in\mathfrak{S}_3^g$ where $(12)$ is the transposition exchanging the two first elements of a set with three ordered elements. We have $l(\tau)=g$. Using theorems 3.4. and 3.5., we see that the weight
$$-2d-\sum\limits_{j=1}^g(2a_j+c_j-1)=p+\sum\limits_{j=1}^g(b_j-a_j+1)=p+g$$
is a weight of the variation of Hodge structure $\mathcal{H}^ni^*j_*\circ\mu_S(V_\lambda)$ for every $g\leq n\leq 2g-1$. 

In particular, using proposition 3.3., we deduce from the previous calculus that $(\partial H^{g+p}(\mathcal{A}^r,\mathbb{Q})\otimes_\mathbb{Q}\mathbb{C})^{e_\lambda}$ has weight $p+g$. Because of theorem 3.2., $\partial M_{gm}(A^r)^{e_\lambda}$ does not avoid weight $0$.

A similar calculus for $((a_i,b_i,c_i)_{1\leq i\leq g},d)$ parallel with $a_i=b_i$ (the case $b_i=c_i$ being treated the same way) shows that the variations of Hodge structure $\mathcal{H}^ni^*j_*\circ\mu_S(V_\lambda)$ has weight $p+3g$ pour $2g\leq n\leq 3g-1$. In particular, because of proposition 3.3., $(\partial H^{p+3g-1}(\mathcal{A}^r,\mathbb{Q})\otimes_\mathbb{Q}\mathbb{C})^{e_\lambda}$ has weight $p+3g$. Using theorem 3.2., we see that the motive $\partial M_{gm}(^\lambda\mathcal{V})$ does not avoid weight $-1$.

\end{proof}

Moreover, we can push the calculus in order to know the weights appearing in $i^*R^kj_*R^pf_*\mathbb{Q}_{\mathcal{A}^r}$, where $f:A^r\rightarrow S$ is the structural morphism. This is done in the following proposition.

\begin{theo}
If $p>6rg$ $i^*R^kj_*Rp^pf_*\mathbb{Q}_{\mathcal{A}^r}$ is zero for every $k$. \\
For $p\leq 6rg$ and $p\neq 1, 6rg-1$, we have \\
(a) for every $0\leq k\leq g-1$, the weights of $i^*R^kj_*R^pf_*\mathbb{Q}_{\mathcal{A}^r}$ are $\{p-mg\}_{0\leq m\leq N_p}$ with 
$$N_p=min([\frac{p}{g}],6r-[\frac{p}{g}], 2r)$$
(b) for every $g\leq k\leq 2g-1$, the weights of$i^*R^kj_*R^pf_*\mathbb{Q}_{\mathcal{A}^r}$ are $\{p-mg\}_{-1\leq m\leq N_p}$ with 
$$N_p=min([\frac{p}{g}],[\frac{p+r(k-g+1)}{k+1}],[\frac{6rg+r(k-g+1)-p}{k+1}],[\frac{6rg-p}{g}])-1$$
(c) for every $2g\leq k\leq 3g-1$, the weights of $i^*R^kj_*R^pf_*\mathbb{Q}_{\mathcal{A}^r}$ are $\{p+mg\}_{3\leq m\leq N_p}$ with 
$$N_p=min([\frac{p}{g}],[\frac{p+r(k-g+1)}{k+1}],[\frac{6rg+r(k-g+1)-p}{k+1}],[\frac{6rg-p}{g}])+3$$
(d) for every $3g\leq k\leq 4g-1$, the weights of $i^*R^kj_*R^pf_*\mathbb{Q}_{\mathcal{A}^r}$ are $\{p+mg\}_{4\leq m\leq N_p}$ with 
$$N_p=min([\frac{p}{g}],6r-[\frac{p}{g}], 2r)+4$$

If $p=1$, $i^*R^kj_*R^1f_*\mathbb{Q}_{\mathcal{A}^r}$ has only weight $1$ for every $1\leq k\leq g$, weight $1+4g$ for $3g-1\leq k\leq 4g-2$ and is trivial for other $k$.

When $p=6rg-1$, $i^*R^kj_*R^{6rg-1}f_*\mathbb{Q}_{\mathcal{A}^r}$ has only weight $6rg-1$ for $1\leq k\leq g$, weight $6rg-1+4g$ for $3g-1\leq k\leq 4g-2$ and is trivial for other $k$.
\end{theo}
\begin{proof}
First, we treat the case $p=1$. Using the decomposition given of $V_\mathbb{C}^{\vee \oplus r}$ with theorems 3.4. and 3.5., we get the result.

We remark the existence of a bijection between characters giving irreducible sub-representations of $\Lambda^p(V_\mathbb{C}^{\vee\oplus r})$ and those giving sub-representations of $\Lambda^{6rg-p}(V_\mathbb{C}^{\vee\oplus r})$ which is
$$((a_i,b_i,c_i)_{1\leq i\leq g},d)\mapsto ((-c_i,-b_i,-a_i)_{1\leq i\leq g},-3rg-d).$$
This bijection permits to deduce the result for $p=6rg-1$ from the one for $p=1$.\newline

We use the notations of theorem 3.6. Let $V_\lambda$ be an irreducible sub-representation of $\Lambda^p(V_\mathbb{C}^{\vee\oplus r})$. When $0\leq k\leq g-1$, the set denoted by $I$ is not empty. Moreover, for $i\in I$ we have $a_i-c_i=m$. Thus, $0\leq m\leq 2r$.

Let us suppose that $p\leq rg$. Let $m$ be an integer such that $0\leq mg\leq p$. There exists a character $((a_i,b_i,c_i)_{1\leq i\leq g},d)$ with $a_i=m$, $b_i=0$ and $c_i=0$ (appart for maybe one $i$, for which $b_i=1$ depending on the parity of $p$) such that $V_\lambda\subset \Lambda^p(V_\mathbb{C}^{\vee\oplus r})$ and $i^*R^kj_*\mu(V_\lambda)$ has weight $p-mg$. Proposition 3.2. shows that $m$ can not be bigger than $[\frac{p}{g}]$.

When $r< m$, $mg\leq p$ and $mg\leq 6rg-p$, there exists a character $((a_i,b_i,c_i)_{1\leq i\leq g},d)$ with $a_i=r$, $b_i=0$ et $c_i=r-m$ (except $b_i=1$ for maybe one $i$ depending on the parity of $p$) such that $V_\lambda\subset\Lambda^p(V_\mathbb{C}^{\vee\oplus r})$ and $i^*R^kj_*\mu(V_\lambda)$ has weight $p-mg$.

This description together with proposition 3.2. shows also that the values of $m$ such that $p-mg$ is a weight of $i^*R^kj_*R^pf_*\mathbb{Q}_{\mathcal{A}^r}$ can not be bigger than $N_p$.\newline

Let $k$ be an integer such that $g\leq k\leq 2g-1$ and $V_\lambda\subset\Lambda^p(V_\mathbb{C}^{\vee\oplus r})$. In this case, one of the two sets denoted $J$ and $K$ is not empty. Without loss of generality, we can suppose $J\neq\emptyset$. Then, there exists $m$ such that $a_j-b_j-1=m$. Thus $-1\leq m$.

Let us suppose that $p\leq rg$. If we choose $0\leq (m+1)g\leq p$, we can find a character $\lambda$ such that $a_i=m+1$, $b_i=0$ et $c_i=0$ (except maybe for one $i$ such that $c_i=-1$ depending of the parity of $p$) with $V_\lambda\subset\Lambda^p(V_\mathbb{C}^{\vee\oplus r})$ and such that $i^*R^kj_*\mu(V_\lambda)$ has weight $p-mg$. Again, proposition 3.2. shows that $m+1\leq [\frac{p}{g}]$.

When $m>r$, $mg\leq N_p$, we choose $k-g+1$ integers $1\leq i\leq g$ such that $a_i=r$, $b_i=c_i=r-m-1$ (unless maybe one $i$ such that $c_i=b_i-1$ depending of the parity of $p$); for the other integers $1\leq i\leq g$, we take $a_i=r$, $b_i=0$ et $c_i=r-m$. Then there exists a character $\lambda=((a_i,b_i,c_i),d)$ with these values of $a_i,b_i,c_i$ such that $V_\lambda\subset\Lambda^p(V_\mathbb{C}^{\vee\oplus r})$ and $i^*R^kj_*\mu(V_\lambda)$ has weight $p-mg$.\newline

The cases $2g\leq k\leq 3g-1$ et $3g\leq k\leq 4g-1$ are treated in a similar way.
\end{proof}

\section{The motive of an automorphic form}

Let $\lambda$ be a character of $T_m\subset G_\mathbb{C}$. We have an action of the Hecke algebra of $S$ on the triangle (see e.g. example 5.4. in \cite{WilBF})
$$\partial M_{gm}(A^r)^{e_\lambda}\rightarrow M_{gm}(A^r)^{e_\lambda}\rightarrow M_{gm}^c(A^r)^{e_\lambda}\rightarrow\partial M_{gm}(A^r)^{e_\lambda}[1].$$

Let us suppose in addition that $\partial M_{gm}(A^r)^{e_\lambda}$ is without weights $0$ and $-1$. As in the case of Picard surfaces (see corollary 3.10. of \cite{Wil14}), the Hecke algebra $\mathfrak{H}(K,G(\mathbb{A}_f))$ associated to the group $K$ acts on $Gr_0M_{gm}(\mathcal{A}^r)^{e_\lambda}$.\newline

In order to define isotypical components of the interior motive, we are going to study the action of the Hecke algebra on the realizations of the interior motive.

Let $\overline{E}$ be an algebraic closure of $E(G,X)$ and $Q$ a field of coefficients containing $E_{gal}$.

\begin{prop}
Under the previous hypotheses, we have isomorphisms
$$(H^n_!(\mathcal{A}^r(\mathbb{C}),\mathbb{Q})\otimes_\mathbb{Q}Q)^{e_\lambda}\simeq (H^{n-p}_!(S(\mathbb{C}),\mu(V_\lambda))$$
$$(H^n_!(\mathcal{A}^r_{\overline{E}},\mathbb{Q}_l)\otimes_\mathbb{Q}Q)^{e_\lambda}\simeq (H^{n-p}_!(S_{\overline{E}},\mu_l(V_\lambda)).$$
\end{prop}
\begin{proof}
It suffices to show a version of proposition 3.3. for the cohomology with compact support. This can be done exactly as in this proposition.
\end{proof}

We consider $R_l$ the l-adic realization and $R_h$ the Hodge realization on $DM_{gm}(E(G,X))_Q$.
\begin{prop}
The realizations $R_l$ and $R_h$ of the $e_\lambda$-part of the interior motive of $A^r$ are
$$R_h(Gr_0M_{gm}(\mathcal{A}^r)^{e_\lambda})=H^{2g}_!(S(\mathbb{C}),\mu(V_\lambda))[-2g-p]$$
$$R_l(Gr_0M_{gm}(\mathcal{A}^r)^{e_\lambda})=H^{2g}_!(S_{\overline{E}},\mu_l(V_\lambda))[-2g-p].$$
\end{prop}
\begin{proof}
Because of theorem 4.7. of \cite{Wil08}, we know that the realization of the interior motive is interior cohomology.

The vanishing theorem of Saper (\cite{Sa}) implies that the interior cohomology of $S$ is non trivial only in degree $2g$.

The results follows the previous proposition.
\end{proof}

%rajouter la dite remarque
The following theorem can be found in the second chapter, at page 62 of \cite{Ha}.

\begin{theo}
Let $Q$ be a field containing $E_{gal}$. Then the $\mathfrak{H}(K,G(\mathbb{A}_f)\otimes_EQ$-module $H^{2g}_!(S(\mathbb{C}),\mu(V_\lambda))$ is semi-simple.
\end{theo}
\begin{proof}
Proof is given in section 4.3.5. of chapter 3 of \cite{Ha}.
\end{proof}

We denote by $R(\mathfrak{H})$ the image by Hodge realization of the Hecke algebra $\mathfrak{H}(G,K)$ in the algebra of endomorphisms of interior cohomology of $S$ in degree $2g$. %(en version Hodge ou l-adique). 

\begin{coro}
The $Q$-algebra $R(\mathfrak{H})\otimes_E Q$ is semi-simple.
\end{coro}
\begin{proof}
This follows from the previous theorem and from proposition 3. of §5.1. at page 46 of \cite{B}.
\end{proof}

As a consequence, the isomorphism classes of simple right $R(\mathfrak{H})\otimes_E Q$-modules are in bijection with the isomorphism classes of minimal right ideals. These ideals are generated by idempotents. 

Let $Y_{\pi_f}$ be such an ideal and let $e_{\pi_f}\in R(\mathfrak{H}(K,G(\mathbb{A}_f))\otimes_E Q$ be the idempotent generating it.

\begin{de}
We define the Hodge structure associated to $Y_{\pi_f}$ as being 
$$W(\pi_f)=Hom_{R(\mathfrak{H}(K,G(\mathbb{A}_f))\otimes_E Q}(Y_{\pi_f},H^{2g}_!(S(\mathbb{C}),\mu(V_\lambda))\otimes_{\overline{E}} Q).$$
%On définit le module Galoisien associé à $Y_{\pi_f}$ de façon analogue :
%$$W(\pi_f)_l=Hom_{R(\mathfrak{H}(K,G(\mathbb{A}_f))\otimes_E Q}(Y_{\pi_f},H^{2g}_!(S_{\overline{E}},\mu_l(V_\lambda))\otimes_E Q)$$
\end{de}

In other words, (see for example \cite{Wil14})
$$W(\pi_f)=H^{2g}_!(S(\mathbb{C}),\mu(V_\lambda))\otimes_{\overline{E}}Q.e_{\pi_f}.$$
%$$W(\pi_f)_l=H^{2g}_!(S_{\overline{E}},\mu_l(V_\lambda))\otimes_{\overline{E}}Q.e_{\pi_f}.$$

A priori, the element $e_{\pi_f}\in R(\mathfrak{H})\otimes_EQ$ does not come from a Chow motive. This is why we need to consider $Gr_0(M_{gm}(A^r)^{e_\lambda})'$ the Grothendieck motive associated to the Chow motive $ Gr_0(M_{gm}(A^r)^{e_\lambda})$. 

The category of Grothendieck motives is the opposite category of the category of homological motives defined in chapter 4 of \cite{And}.

\begin{de}
We define the motive associated to $Y_{\pi_f}$ as being the image by $e_{\pi_f}$ of the Grothendieck motive $Gr_0(M_{gm}(A^r)^{e_\lambda})'$ associated to the $e_\lambda$-part of the interior motive of $A^r$. We write it $\mathcal{W}(\pi_f)$.
\end{de}

\begin{theo}
The realizations of $\mathcal{W}(\pi_f)$ are concentrated in cohomological degree $p+2g$ and are equal to $W(\pi_f)$ in Hodge realization. %et à $W(\pi_f)_l$ en réalisation l-adique.
\end{theo}
\begin{proof}
This follows from the construction.
\end{proof}

\bibliographystyle{plain}
\bibliography{bibart}
\nocite{*}
\end{document}